\documentclass[a4paper,oneside,10pt]{article} 

\title{ARTICLE - A CURIE-WEISS MODEL OF SELF-ORGANIZED CRITICALITY - GAUSSIAN CASE}
\author{Rapha\"el Cerf & Matthias Gorny}
\date{\textit{MAJ: 04 juin 2013}}

\usepackage[english,frenchb]{babel}
\usepackage[latin1]{inputenc}
\usepackage[T1]{fontenc}
\usepackage{amsfonts}
\usepackage{graphicx}
\usepackage[cyr]{aeguill}
\usepackage{xspace}
\usepackage{tikz}
\usepackage{cite}
\usepackage{tocloft}
\usepackage{amsthm}

\usepackage{amssymb}
\usepackage{amsmath}
\usepackage{stmaryrd}
\usepackage{latexsym}
\usepackage{exscale}
\usepackage{relsize}
\usepackage{accents}
\usepackage{dsfont}
\usepackage{mathrsfs}
\usepackage{yhmath}

\newcommand{\ind}[1]{\mathds{1}_{#1}}

\newcommand{\R}{\mathbb{R}}
\newcommand{\C}{\mathbb{C}}

\newcommand{\E}{\mathbb{E}}

\newcommand{\Bl}{\mathrm{B}}

\newcommand{\Lc}{\mathcal{L}}

\newcommand{\Ck}[1]{\mathcal{C}^{#1}}

\renewcommand{\Re}{\mathfrak{Re}}
\renewcommand{\Im}{\mathfrak{Im}}

\def\Dro{\smash{{D}^{\!\!\!\!\raise4pt\hbox{$\scriptstyle o$}}}}
\def\Ccro{\smash{{\mathcal{C}}^{\!\!\!\raise4pt\hbox{$\scriptstyle o$}}}}
\def\Aro{\smash{{A}^{\!\!\!\raise5pt\hbox{$\scriptstyle o$}}}}

\renewcommand{\a}{\alpha}
\renewcommand{\b}{\beta}
\newcommand{\g}{\gamma}
\newcommand{\G}{\Gamma}
\renewcommand{\d}{\delta}

\renewcommand{\epsilon}{\varepsilon}
\newcommand{\eps}{\varepsilon}
\newcommand{\z}{\zeta}
\renewcommand{\r}{\rho}
\newcommand{\s}{\sigma}

\renewcommand{\th}{\theta}

\renewcommand{\phi}{\varphi}

\renewcommand{\O}{\Omega}

\newtheorem{theo}{Theorem}

\newtheorem{prop}[theo]{Proposition}

\newtheorem{lem}[theo]{Lemma}

\renewenvironment{proof}{\noindent{\bf Proof.}}{\qed}

\begin{document}

\renewcommand{\contentsname}{Contents}
\renewcommand{\refname}{\textbf{References}}
\renewcommand{\abstractname}{Abstract}

\begin{center}
\begin{Huge}
A Curie-Weiss Model

of Self-Organized Criticality :\medskip

The Gaussian Case
\end{Huge} \bigskip \bigskip \bigskip \bigskip

\begin{Large} Matthias Gorny \end{Large} \smallskip
 
\begin{large} {\it Universit\'e Paris Sud} \end{large} \bigskip \bigskip %\bigskip

%\begin{large} June 4, 2013 \end{large}
\end{center}
\bigskip \bigskip \bigskip\bigskip \bigskip

\begin{abstract}
\noindent We try to design a simple model exhibiting self-organized criticality, which is amenable to a rigorous mathematical analysis. To this end, we modify the generalized Ising Curie-Weiss model by implementing an automatic control of the inverse temperature. With the help of exact computations, we show that, in the case of a centered Gaussian measure with positive variance $\s^{2}$, the sum $S_n$
of the random variables has fluctuations of order $n^{3/4}$ and that $S_n/n^{3/4}$ converges to the distribution $C \exp(-x^{4}/(4\s^4))\,dx$ where $C$ is a suitable positive constant.
\end{abstract}
\bigskip \bigskip \bigskip \bigskip \bigskip\bigskip \bigskip



\newpage

\section{Introduction}

\noindent In their famous article~\cite{BTW}, Per Bak, Chao Tang and Kurt Wiesenfeld showed that certain complex systems are naturally attracted by critical points, without any external intervention. These systems exhibit the phenomenon of self-organized criticality. It can be observed empirically or simulated on a computer in various models. However the mathematical analysis of these models turns out to be extremely difficult.\medskip

\noindent Our goal here is to design a model exhibiting self-organized criticality, which is as simple as possible, and which is amenable to a rigorous mathematical analysis. The idea is to start with the Ising Curie-Weiss model (see~\cite{Ellis}), which presents a phase transition, and to create a feedback from the configuration to the control parameters in order to converge towards a critical point.\medskip

\noindent The generalized Ising Curie-Weiss model (see~\cite{EN}) associated to a probability measure $\r$ in $\R$ (with some \og sub-Gaussian \fg{} conditions) and the inverse temperature $\b>0$ is defined through an infinite triangular array of real-valued random variables $(X_{n}^{k})_{1\leq k \leq n}$ such that, for all $n \geq 1$, $(X^{1}_{n},\dots,X^{n}_{n})$ has the distribution
\[\frac{1}{Z_{n}(\b)} \exp\left(\frac{\b}{2}  \frac{(x_{1}+\dots+x_{n})^{2}}{n}\right) \prod_{i=1}^{n} d\r(x_{i})\]
where $Z_{n}(\b)$ is a normalization. For any $n \geq 1$, we set $S_{n}=X^{1}_{n}+\dots+X^{n}_{n}$. Let $\s^2$ be the variance of $\r$. Ellis and Newman~\cite{EN} have proved the following result. If $\b<1/\s^{2}$, then the fluctuations of $S_{n}$ are of order $\sqrt{n}$ and $S_{n}/\sqrt{n}$ converges towards a specific Gaussian distribution. If $\b=1/\s^{2}$, then the fluctuations of $S_{n}$ are of order $n^{1-1/2k}$, where $k$ is an integer depending on the distribution $\r$. The point $1/\s^{2}$ is the critical value of the generalized Ising Curie-Weiss model.\medskip

\noindent In order to obtain a model which presents self-organized criticality we transform the previous probability distribution by \og replacing $\b$ by $n\,(x_{1}^{2}+\dots+x_{n}^{2})^{-1}$ \fg. Hence the model we consider is given by the distribution
\[\frac{1}{Z_{n}} \exp\left(\frac{1}{2}\frac{(x_{1}+\dots+x_{n})^{2}}{x_{1}^{2}+\dots+x_{n}^{2}}\right)\prod_{i=1}^{n} d\r(x_{i})\]
We refer to~\cite{CerfGorny} for a more detailed explanation. This model can be defined for any distribution $\r$, in particular for any Gaussian measure (contrary
to the generalized Ising Curie-Weiss model). In this note, we consider the case where $\r$ is the centered Gaussian measure with variance $\s^2$. With the help of exact computations, we show that $S_n/n^{3/4}$ converges to the distribution:
\[\left(\frac{\s}{\sqrt{2}}\,\,\G\left(\frac{1}{4}\right)\right)^{-1}\exp\left(-\frac{x^{4}}{4\s^{4}}\right)\,dx\]

\noindent The computations we make here are not possible for more general probability measures. In~\cite{CerfGorny} we consider a class of distributions having an even density with respect to the Lebesgue measure and satisfying some integrability conditions and we prove a similar convergence result.
\medskip

\noindent In section~\ref{model} we define properly our model for Gaussian measures and we state our main result. The proof is split in the two remaining sections.

\section{Main result}
\label{model}

\noindent We denote by $\r_{\s}$ the Gaussian distribution with mean $0$ and variance $\s^{2}>0$. We consider $(X_{n}^{k})_{1\leq k \leq n}$ an infinite triangular array of real-valued random variables such that for all $n \geq 1$, $(X^{1}_{n},\dots,X^{n}_{n})$ has the distribution $\widetilde{\mu}_{n,\s}$, where 
\begin{multline*}
d\widetilde{\mu}_{n,\s}(x_{1},\dots,x_{n})=\frac{1}{Z_{n}}\exp\left(\frac{1}{2}\frac{(x_{1}+\dots+x_{n})^{2}}{x_{1}^{2}+\dots+x_{n}^{2}}\right)\,\prod_{i=1}^{n}d\r_{\s}(x_{i})\\
=\frac{1}{(2\pi\s^{2})^{n/2}Z_{n}}\exp\left(\frac{1}{2}\frac{(x_{1}+\dots+x_{n})^{2}}{x_{1}^{2}+\dots+x_{n}^{2}}-\frac{x_{1}^{2}+\dots+x_{n}^{2}}{2\s^{2}}\right)\,\prod_{i=1}^{n}dx_{i}
\end{multline*}
and
\[Z_{n}=\int_{\R^{n}}(2\pi)^{-n/2}\exp\left(\frac{1}{2}\frac{(x_{1}+\dots+x_{n})^{2}}{x_{1}^{2}+\dots+x_{n}^{2}}-\frac{1}{2}(x_{1}^{2}+\dots+x_{n}^{2})\right)\,\prod_{i=1}^{n}dx_{i}\]
We define $S_{n}=X^{1}_{n}+\dots+X^{n}_{n}$ and $T_{n}=(X^{1}_{n})^{2}+\dots+(X^{n}_{n})^{2}$.
\medskip

\noindent We notice that the event $\{x_{1}^{2}+\dots+x_{n}^{2}=0\}$ is negligible for the measure $\r_{\s}^{\otimes n}$, so that the denominator in the exponential is almost surely positive. Moreover, $t \longmapsto t^{2}$ is a convex function, thus for any $n \geq 1$, $1 \leq Z_{n} \leq e^{n/2}< +\infty$.
\medskip

\begin{theo} Under $\widetilde{\mu}_{n,\s}$, $(S_{n}/n,T_{n}/n)$ converges in probability towards $(0,\s^{2})$. Moreover
\[\frac{S_{n}}{ n^{3/4}} \overset{\Lc}{\underset{n \to +\infty}{\longrightarrow}} \left(\frac{\s}{\sqrt{2}}\,\,\G\left(\frac{1}{4}\right)\right)^{-1}\exp\left(-\frac{y^{4}}{4\s^{4}}\right)\,dx\]
\label{theoCVgauss}
\end{theo}

\noindent To prove this theorem, we first compute, in section~\ref{lawSn/n}, the exact density of the law of $(S_{n},T_{n})$ under $\widetilde{\mu}_{n,\s}$, for $n$ large enough. Next, in section~\ref{ProofCVgauss}, we end the proof by using Laplace's method.

\section{Computation of the law of $(S_{n},T_{n})$}
\label{lawSn/n}

\noindent In this section we compute the law of $(S_{n},T_{n})$ under $\widetilde{\mu}_{n,\s}$.

\begin{lem} We denote by $\nu_{\s}$ the law of $(Z,Z^{2})$ where $Z$ is a Gaussian random variable with mean $0$ and variance $\s^{2}>0$. Under $\widetilde{\mu}_{n,\s}$, the law of $(S_{n},T_{n})$ is
\[\frac{1}{Z_{n}}\exp\left(\frac{x^{2}}{2y}\right)\,d\nu_{\s}^{*n}(x,y)\]
\label{loiSnTn1}
\end{lem}
\vspace{-0.3cm}

\begin{proof} Let $f : \R^{2}\longrightarrow \R$ be a bounded measurable function. We have
\begin{multline*}
\E_{\tilde{\mu}_{n,\s}}(f(S_{n},T_{n}))=\frac{1}{Z_{n}}\int_{\R^{n}}f(x_{1}+\dots+x_{n},x_{1}^{2}+\dots+x_{n}^{2})\\
\hfill \exp\left(\frac{1}{2}\frac{(x_{1}+\dots+x_{n})^{2}}{x_{1}^{2}+\dots+x_{n}^{2}}\right) \,\prod_{i=1}^{n}d\r_{\s}(x_{i})
\end{multline*}
The function $h : (x,y) \in \R\times \R\backslash\{0\} \longmapsto f(x,y)\exp(x^{2}/(2y))$ is measurable. Therefore
\begin{multline*}
\E_{\tilde{\mu}_{n,\r}}(f(S_{n},T_{n}))=\frac{1}{Z_{n}}\int_{\R^{n}}h(x_{1}+\dots+x_{n},x_{1}^{2}+\dots+x_{n}^{2})\,\prod_{i=1}^{n}d\r_{\s}(x_{i})\\
=\frac{1}{Z_{n}}\int_{\R^{2n}}h(z_{1}+\dots+z_{n})\prod_{i=1}^{n}d\nu_{\s}(z_{i})=\frac{1}{Z_{n}}\int_{\R^{2}}h(z)\,d\nu_{\s}^{*n}(z)
\end{multline*}
Hence the announced law of $(S_{n},T_{n})$, under $\widetilde{\mu}_{n,\s}$.
\end{proof}
\medskip

\noindent We denote by $\G$ the gamma function defined by
\[\forall z>0 \qquad \G(z)=\int_{0}^{+\infty}x^{z-1}e^{-x}dx\]
We compute next the density of $\nu_{\s}^{*n}$:

\begin{prop} For $n\geq 5$, under $\widetilde{\mu}_{n,\s}$, the law of $(S_{n},T_{n})$ is
\[\frac{1}{\s^{n}C_{n}}\exp\left(\frac{x^{2}}{2y}-\frac{y}{2\s^{2}}\right)\left(y-\frac{x^{2}}{n}\right)^{(n-3)/2}\ind{x^{2}< n y}\,dx\,dy\]
where $C_{n}=Z_{n}\,\sqrt{2^{n}\pi n}\,\G((n-1)/2)$.
\label{loiSnTn2}
\end{prop}

\noindent For simplicity, we assume that $\s^{2}=1$. We just write $\nu^{*n}$ for $\nu_{\s}^{*n}$. We denote by $\Phi_{n}$ its characteristic function. To get the previous proposition, we use the method of residue to compute $\nu^{*n}$ and a Fourier inversion formula to get the density of $\nu^{*n}$. For $(u,v) \in \R^{2}$, we have
\[\Phi_{n}(u,v)=\left(\Phi_{1}(u,v)\right)^{n}=\left(\E(e^{iuZ+ivZ^{2}})\right)^{n}=\left(\int_{\R}e^{iux+ivx^{2}}e^{-x^{2}/2}\,\frac{dx}{\sqrt{2 \pi}}\right)^{n}\]
We need some preliminary results:
\medskip

\noindent The Gamma distribution with shape $k>0$ and scale $\theta>0$, denoted by $\G(k,\th)$, is the probability distribution with density function
\[x \longmapsto \frac{x^{k-1}e^{-x/\theta}}{\G(k)\, \theta^{k}}\ind{x>0}\]
with respect to the Lebesgue measure on $\R$.
\medskip

\noindent The complex logarithm function (or the principle value of complex logarithm), denoted by $\mathrm{Log}$, is defined on $\O=\C\backslash]-\infty,0]$ by
\[\forall z=x+iy \in \O \qquad \mathrm{Log}(z)=\frac{1}{2}\ln(x^{2}+y^{2})+2i\,\mathrm{arctan}\left(\frac{y}{x+\sqrt{x^{2}+y^{2}}}\right)\]
If $\a\in \C$ and $z \in \O$, then the $\a$-exponentiation of $z$ is defined by 
\[z^{\a}=\exp(\a \mathrm{Log}(z))\]
By chapter XV of~\cite{Feller2}, for $k,\theta>0$, the characteristic function of $\G(k,\th)$ is
\[u \in \R \longmapsto(1-\theta i u)^{-k}\]

\noindent We can now prove the following key lemma:

\begin{lem} Let $t \in \R$ and $\z \in \C$ such that $\Re(\z)>0$. Then
\[\int_{\R}e^{itx-\z x^{2}/2}\,dx=\sqrt{\frac{2 \pi}{\Re(\z)}}\exp\left(-\frac{t^{2}}{2\z}\right)\left(1+i\,\frac{\Im(\z)}{\Re(\z)}\right)^{-1/2}\]
\label{lemclegauss}
\end{lem}

\begin{proof} Let $t \in \R$ and $\z=a+ib \in \C$ such that $\Re(\z)>0$. We define
\[K(t,\z)=\int_{\R}e^{itx-\z x^{2}/2}\,dx\]
We factorize: 
\[ixt-\frac{1}{2}\z x^{2}=-\frac{1}{2}\z\left(x-\frac{it}{\z}\right)^{2}-\frac{t^{2}}{2\z}=-\frac{1}{2}\z\left(x-\frac{tb}{|\z|}-i\frac{ta}{|\z|}\right)^{2}-\frac{t^{2}}{2\z}\] 
Thus
\[e^{t^{2}/2\z}K(t,\z)=\int_{\R}e^{-\z(x-tb/|\z|-ita/|\z|)^{2}/2}\,dx\]
The change of variables $y=x-tb/|\z|$ gives us
\[e^{t^{2}/2\z}K(t,\z)=\int_{\R}e^{-\z(y-ita/|\z|)^{2}/2}\,dy=-\lim_{R \to +\infty}\int_{\g_{1}}e^{-\z z^{2}/2}\,dz\]
where the last integral is the contour integral of the entire function $z\longmapsto e^{-\z z^{2}/2}$, along the segment $\g_{1}$ in the complex plane with end points $R+ita/|\z|$ and $-R+ita/|\z|$.
\medskip

\noindent Let $\g$ be the rectangle in the complex plane joining successively the points $R+ita/|\z|$, $-R+ita/|\z|$, $-R$ and $R$. We apply the residue theorem: 
\[\int_{\g}e^{-\z z^{2}/2}\,dz=0\]
since $z\longmapsto \exp(-\z z^{2}/2)$ has no pole (see~\cite{Rudin}). We denote $\g_{1}$, $\g_{2}$, $\g_{3}$ and $\g_{4}$ the successive edges of the rectangle $\g$.

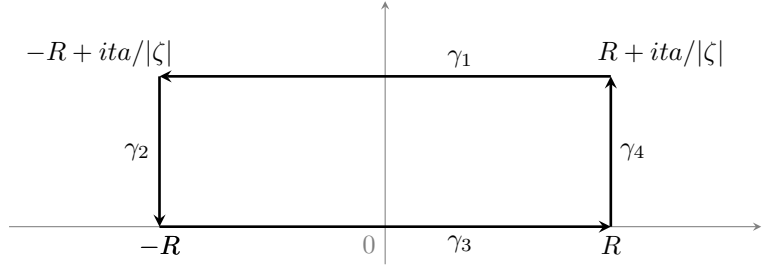
\begin{center}
\begin{tikzpicture}
\draw [>=stealth,->] [color=gray!100](-5,0) -- (5,0) ;
\draw [>=stealth,->] [color=gray!100](0,-0.5) -- (0,3) ;
\draw [>=stealth,->] [line width=1pt](-3,0) -- (3,0) ;
\draw [>=stealth,->] [line width=1pt](3,0) -- (3,2) ;
\draw [>=stealth,->] [line width=1pt](3,2) -- (-3,2) ;
\draw [>=stealth,->] [line width=1pt](-3,2) -- (-3,0) ;
\draw [color=gray!100](0,0) node[below left] {$0$};
\draw (-3,0) node[below] {$-R$};
\draw (3,0) node[below] {$R$};
\draw (-3,0) node[below] {$-R$};
\draw (2.7,2) node[above right] {$R+ita/|\z|$};
\draw (-2.7,2) node[above left] {$-R+ita/|\z|$};
\draw (1,2) node[above] {$\g_{1}$};
\draw (1,0) node[below] {$\g_{3}$};
\draw (-3,1) node[left] {$\g_{2}$};
\draw (3,1) node[right] {$\g_{4}$};
\end{tikzpicture}
\end{center}
\[\int_{\g_{3}}e^{-\z z^{2}/2}\,dz=\int_{-R}^{R}e^{-\z x^{2}/2}\,dx \underset{R \to +\infty}{\longrightarrow} \int_{\R}e^{-\z x^{2}/2}\,dx=2\int_{0}^{+\infty}e^{-\z x^{2}/2}\,dx\]
We make the change of variables $y=x^{2}$ on $]0,+\infty[$:
\begin{align*}
2\int_{0}^{+\infty}e^{-\z x^{2}/2}\,dx=\int_{0}^{+\infty}e^{-\z y/2}\,\frac{dy}{\sqrt{y}}&=\int_{0}^{+\infty}e^{-iby/2}e^{-ay/2}\,\frac{dy}{\sqrt{y}}\\
&=\sqrt{\frac{2}{a}}\,\G\left(\frac{1}{2}\right)\left(1+i\frac{b}{a}\right)^{-1/2}
\end{align*}
since we recognize, up to a normalization factor, the characteristic function of the Gamma distribution with shape $1/2$ and scale $2/a$. Moreover we have
\begin{align*}
\left|\int_{\g_{4}}e^{-\z z^{2}/2}\,dz\right|&=\left|\int_{0}^{1}\exp\left(-\frac{\z}{2}\left(R+\frac{iat}{|\z|}x\right)^{2}\right)\frac{iat}{|\z|}\,dx\right|\\
&\leq \frac{a|t|}{|\z|} \int_{0}^{1}\exp\left(-\frac{aR^{2}}{2}+\frac{Ratbx}{|\z|}+\frac{a}{2}\left(\frac{atx}{|\z|}\right)^{2}\right)\,dx\\
&\leq \frac{a|t|}{|\z|} \exp\left(-\frac{aR^{2}}{2}+\frac{Ra|tb|}{|\z|}+\frac{a}{2}\left(\frac{at}{|\z|}\right)^{2}\right)\underset{R \to +\infty}{\longrightarrow} 0
\end{align*}
Likewise
\[\int_{\g_{2}}e^{-\z z^{2}/2}\,dz\underset{R \to +\infty}{\longrightarrow} 0\]
Letting $R$ go to $+\infty$, we conclude that
\[\sqrt{\frac{2}{a}}\,\G\left(\frac{1}{2}\right)\left(1+i\frac{b}{a}\right)^{-1/2}+0-e^{t^{2}/2\z}K(t,\z)+0=0\]
Since $\G(1/2)=\sqrt{\pi}$, we obtain the identity stated in the lemma.
\end{proof}
\medskip

\noindent For $(u,v) \in \R^{2}$, setting $\z=1-2iv \in \{\,z \in \C : \Re(z)>0\,\}$, we have
\[\Phi_{n}(u,v)=\frac{1}{(2 \pi)^{n/2}}\left(\int_{\R}e^{iux-\z x^{2}/2}\,dx\right)^{n}\]
Applying lemma~\ref{lemclegauss} with $u$ and $\z$, we obtain the following proposition:

\begin{prop} The characteristic function $\Phi_{n}$ of the distribution $\nu^{*n}$ is
\[(u,v)\in \R^{2}\longmapsto \exp\left(-\frac{n}{2}\left(\frac{u^{2}}{1-2iv}+\mathrm{Log}(1-2iv)\right)\right)\]
\label{phi_n}
\end{prop}

\noindent Once we know the characteristic function $\Phi_{n}$ of the law $\nu^{*n}$, a Fourier inversion formula gives us its density. We first have to check that $\Phi_{n}$ is integrable with respect to the Lebesgue measure on $\R^{2}$.
\medskip

\noindent Let $(u,v) \in \R^{2}$. Since $(1-2iv)^{-1}=(1+2iv)/(1+4v^{2})$, we have
\[\Re\left(\frac{u^{2}}{1-2iv}+\mathrm{Log}(1-2iv)\right)=\frac{u^{2}}{1+4v^{2}}+\ln(\sqrt{1+4v^{2}})\]
Using Fubini's theorem, it follows that
\begin{align*}
\int_{\R^{2}}\left| \Phi_{n}(u,v) \right| \,du\,dv
&=\int_{\R^{2}}\exp\left(-\frac{nu^{2}}{2(1+4v^{2})}\right)(1+4v^{2})^{-n/4}  \,du\,dv\\
&=\int_{\R} (1+4v^{2})^{-n/4} \left(\int_{\R}\exp\left(-\frac{nu^{2}}{2(1+4v^{2})}\right)\,du \right) \,dv\\
&=\int_{\R} (1+4v^{2})^{-n/4} \sqrt{\frac{2 \pi(1+4v^{2})}{n}} \,dv\\
&=\sqrt{\frac{2 \pi}{n}} \int_{\R} (1+4v^{2})^{-(n-2)/4} \,dv
\end{align*}
The function $v\longmapsto (1+4v^{2})^{-(n-2)/4}$ is continuous on $\R$ and integrable in the neighbourhood of $+\infty$ and $-\infty$ if and only if $n>4$.

\begin{prop} If $n\geq 5$ then $\nu^{*n}$ has the density 
\[(x,y)\in \R^{2}\longmapsto \left(\sqrt{2^{n} \pi n}\,\G\left(\frac{n-1}{2}\right)\right)^{-1} \exp\left(-\frac{y}{2}\right)\left(y-\frac{x^{2}}{n}\right)^{(n-3)/2}\ind{x^{2}< n y}\]
with respect to the Lebesgue measure on $\R^{2}$.
\label{densgauss}
\end{prop}

\begin{proof} We have seen that, if $n \geq 5$, then $\Phi_{n}$ is integrable on $\R^{2}$. The Fourier inversion formula (see~\cite{Rudin}) implies that $\nu_{\s}^{*n}$ has the density
\[f_{n} : (x,y)\longmapsto\frac{1}{(2 \pi)^{2}}\int_{\R^{2}}e^{-ixu-iyv}\,\Phi_{n}(u,v)\,du\,dv\]
with respect to the Lebesgue measure on $\R^{2}$. Let $(x,y) \in \R^{2}$. By Fubini's theorem,
\begin{align*}
f_{n}(x,y)&=\frac{1}{(2 \pi)^{2}}\int_{\R}\frac{e^{-iyv}}{(1-2iv)^{n/2}} \left(\int_{\R} \exp\left(-ixu-\frac{nu^{2}}{2(1-2iv)}\right)\,du\right)\,dv\\
&=\frac{1}{(2 \pi)^{2}}\int_{\R}\frac{e^{-iyv}}{(1-2iv)^{n/2}} \, K\left(-x,\frac{n}{1-2iv}\right)\,dv
\end{align*}
where $K$ is defined by
\[\forall a>0 \qquad \forall (t,b) \in \R^{2} \qquad K(t,a+ib)=\int_{\R}e^{itz-(a+ib) z^{2}/2}\,dz\]
Lemma~\ref{lemclegauss} implies that for any $v \in \R$,
\begin{align*}
K\left(-x,\frac{n}{1-2iv}\right)&=\sqrt{\frac{2 \pi (1+4v^{2})}{n}}\exp\left(-\frac{x^{2}(1-2iv)}{2n}\right)\left(1+2iv\right)^{-1/2}\\
&=\sqrt{\frac{2\pi}{n}}\exp\left(-\frac{x^{2}(1-2iv)}{2n}\right) \left(\frac{1+4v^{2}}{1+2iv}\right)^{1/2}\\
&=\sqrt{\frac{2\pi}{n}}\exp\left(-\frac{x^{2}(1-2iv)}{2n}\right) (1-2iv)^{1/2}
\end{align*}
Thus
\begin{align*}
f_{n}(x,y)&= \frac{1}{(2 \pi)^{2}} \sqrt{\frac{2\pi}{n}} \int_{\R}\exp\left(-iyv-\frac{x^{2}(1-2iv)}{2n}\right)(1-2iv)^{-(n-1)/2}\,dv\\
&=\frac{1}{\sqrt{2 \pi n}}  \exp\left(-\frac{x^{2}}{2n}\right) \frac{1}{2\pi}\int_{\R}\exp\left(-iv\left(y-\frac{x^{2}}{n}\right)\right)(1-2iv)^{-(n-1)/2}\,dv
\end{align*}
Therefore $\sqrt{2 \pi n} \exp(x^{2}/2n) f_{n}(x,y)$ is the inverse Fourier transform of the distribution $\G((n-1)/2,2)$ taken at the point $y-x^{2}/n$. Hence
\begin{multline*}
\sqrt{2 \pi n}\exp\left(\frac{x^{2}}{2n}\right)f_{n}(x,y)= \left(\G\left(\frac{n-1}{2}\right)2^{(n-1)/2}\right)^{-1} \left(y-\frac{x^{2}}{n}\right)^{(n-3)/2} \\ \times \exp\left(-\frac{y}{2}+\frac{x^{2}}{2n}\right) \ind{y>x^{2}/n}
\end{multline*}
Simplifying this expression, we get the proposition.
\end{proof}
\medskip

\noindent This previous result and proposition~\ref{loiSnTn1} imply that, for $n \geq 5$, under $\widetilde{\mu}_{n,\s}$, the law of $(S_{n},T_{n})$ on $\R^{2}$ is
\[C_{n}^{-1}\exp\left(\frac{x^{2}}{2y}-\frac{y}{2}\right)\left(y-\frac{x^{2}}{n}\right)^{(n-3)/2}\ind{x^{2}< n y}\,dx\,dy\]
We observe next that $(\s X_{n}^{1},\dots,\s X_{n}^{n})$ has the distribution $\widetilde{\mu}_{n,\s}$ if and only if $(X_{n}^{1},\dots,X_{n}^{n})$ has the distribution $\widetilde{\mu}_{n,1}$. Hence a straightforward change of variables gives us proposition~\ref{loiSnTn2}.

\section{Proof of theorem~\ref{theoCVgauss}}
\label{ProofCVgauss}

\noindent Let $\a,\b \in \, ]0,1]$, $n \geq 5$ and $f$ a bounded measurable function. The change of variables $(x,y)\longmapsto(n^{\a}x,n^{\b}y)$ yields
\begin{multline*}
\E_{\tilde{\mu}_{n,1}}\left(f\left(\frac{S_{n}}{n^{\a}},\frac{T_{n}}{n^{\b}}\right)\right)=\frac{n^{\a+\b}}{C_{n}}\int_{\R^{2}}f(x,y)\exp\left(\frac{n^{2\a-\b}x^{2}}{2y}-\frac{n^{\b}y}{2}\right)\\
\times \left(n^{\b}y-n^{2\a-1}x^{2}\right)^{(n-3)/2}\ind{n^{2\a}x^{2}< n^{\b+1} y}\,dx\,dy
\end{multline*}
Factorizing by $n^{(n-3)/2}$, we notice that all the terms in the integral are functions of $x^{2}/n^{2-2\a}$ and $y/n^{1-\b}$. We obtain the following proposition.

\begin{prop} Let $\a,\b \in \, ]0,1]$. If $\s^{2}=1$ and $n \geq 5$ then, under $\widetilde{\mu}_{n,\s}$, the distribution of $(S_{n}/n^{\a},T_{n}/n^{\b})$ is
\[\frac{n^{\a+\b}n^{(n-3)/2}}{C_{n}}\exp\left(-n\psi\left(\frac{x^{2}}{n^{2-2\a}},\frac{y}{n^{1-\b}}\right)\right)\phi\left(\frac{x^{2}}{n^{2-2\a}},\frac{y}{n^{1-\b}}\right)\,dx\,dy\]
where $\psi$ and $\phi$ are the functions defined on $D^{\!+}=\{\,(x,y) \in \R^{2} : y>x \geq 0\,\}$ by 
\[\psi : (x,y) \longmapsto \frac{1}{2}\left(-\frac{x}{y}+y-\ln(y-x)\right)\]
\[\phi : (x,y) \longmapsto (y-x)^{-3/2}\,\ind{D^{\!+}}(x,y)\]
\label{loiSnTn3}
\end{prop}
\vspace{-0.5cm}

\noindent We give next some properties of the map $\psi$. Especially they show why we choose $\a=3/4$ and $\b=1$ in the previous proposition in order to prove theorem~\ref{theoCVgauss}.

\begin{lem} The map $\psi$ has a unique minimum at $(0,1)$ and, in the neighbourhood of $(0,1)$,
\[\psi(x,y)-\frac{1}{2}=\frac{1}{4}(x^{2}+(y-1)^{2})+o(\|x,y-1\|^{2})\]
Moreover, we have 
\[\forall \d>0 \qquad \inf\,\{\,\psi(x,y):|x|\geq\d \quad\mbox{or}\quad |y-1|\geq\d\,\}>1/2\]
\label{lemGH}
\end{lem}
\vspace{-0.5cm}

\begin{proof} The map $\psi$ is $\Ck{2}$ on $D^{\!+}$ and, for fixed $y>0$, \[\frac{\partial \psi}{\partial x}(x,y)=\frac{1}{2}\left(-\frac{1}{y}+\frac{1}{y-x}\right)\geq 0\]
Equality holds if and only if $x=0$. Thus $x\longmapsto \psi(x,y)$ is increasing on $]0,y[$ and $\psi(0,y)=(y-\ln(y))/2$. Hence for any $(x,y) \in D^{\!+}\backslash\{(0,1)\}$,
\[\psi(x,y) > \frac{1}{2}(y-\ln(y))>\frac{1}{2}=\psi(0,1) \]
Therefore $\psi$ has a unique minimum at $(0,1)$. In the neighbourhood of $(0,0)$,
\begin{align*}
\psi(x,1+h)&=\frac{1}{2}(-x(1-h+o(h^{2}))+1+h-(h-x-\frac{1}{2}(h-x)^{2}+o((h-x)^{2})\\
&=\frac{1}{2}+\frac{h^{2}}{4}+\frac{x^{2}}{4}+o(\|x,h\|^{2})
\end{align*}
Hence the announced expansion of $\psi$ in the neighbourhood of $(0,1)$.
Moreover, if $|y-1|\geq \d$ and $x \in \,[0,y[$, then
\[\psi(x,y)\geq \frac{1}{2}(1+\d-\ln(1+\d))>\frac{1}{2}\]
If $x \geq \d$ and $y > x$, then
\[2\psi(x,y)\geq -\frac{\d}{y}+y-\ln(y-\d)>\inf_{y >\d}\left(-\frac{\d}{y}+y-\ln(y-\d)\right)>1\]
since $\d \neq 0$. Therefore $\inf\,\{\,\psi(x,y):|x|\geq\d \quad\mbox{or}\quad |y-1|\geq\d\,\}>1/2$.
\end{proof}
\medskip

\noindent By this lemma, for fixed $(x,y)$, when $n$ goes to $+\infty$,
\[\psi\left(\frac{x^{2}}{n^{2-2\a}},\frac{y}{n^{1-\b}}\right)-\frac{1}{2}\sim \frac{x^{4}}{4} n^{3-4\a}+\frac{n}{4}\left(\frac{y}{n^{1-\b}}-1\right)^{2}\]
That is why we take $\a=3/4$ and $\b=1$.
\medskip

\noindent Let us prove theorem~\ref{theoCVgauss}. Let $n\geq 1$ and let $f : \R^{2} \longrightarrow \R$ be a continuous bounded function. By proposition~\ref{loiSnTn3}, we have
\begin{multline*}
\E_{\tilde{\mu}_{n,1}}\left(f\left(\frac{S_{n}}{n^{3/4}},\frac{T_{n}}{n}\right)\right)=\frac{n^{7/4}n^{(n-3)/2}}{C_{n}}\int_{\R^{2}}f(x,y)\exp\left(-n\psi\left(\frac{x^{2}}{\sqrt{n}},y\right)\right)\\
\times \phi\left(\frac{x^{2}}{\sqrt{n}},y\right) \ind{\sqrt{n}y>x^{2}}\,dx\,dy
\end{multline*}
It follows from the expansion of $\psi$ in lemma~\ref{lemGH} that there exists $\d>0$ such that for $(x,y)\in D^{\!+}$, if $|x|<\d$ and $|y-1|<\d$, then, 
\[\psi(x,y)-\frac{1}{2}\geq \frac{1}{8}(x^{2}+(y-1)^{2})\]
We denote
\[A_{n}=\int_{x^{2}<\d \sqrt{n}}\int_{|y-1|<\d}\!f(x,y)\exp\left(-n\psi\left(\frac{x^{2}}{\sqrt{n}},y\right)\right)\phi\left(\frac{x^{2}}{\sqrt{n}},y\right) \ind{\sqrt{n}y>x^{2}}\,dx\,dy\]
The change of variables $(x,y) \longmapsto (x,y/\sqrt{n}+1)$ gives
\begin{multline*}
\sqrt{n}e^{n/2}A_{n}=\int_{x^{2}<\d \sqrt{n}}\int_{|y|<\d\sqrt{n}}f\left(x,\frac{y}{\sqrt{n}}+1\right)\exp\left(-n\psi\left(\frac{x^{2}}{\sqrt{n}},\frac{y}{\sqrt{n}}+1\right)\right)\\\exp\left(\frac{n}{2}\right)\phi\left(\frac{x^{2}}{\sqrt{n}},\frac{y}{\sqrt{n}}+1\right) \ind{y+\sqrt{n}>x^{2}}\,dx\,dy
\end{multline*}
Lemma~\ref{lemGH} implies that
\[n\psi\left(\frac{x^{2}}{\sqrt{n}},\frac{y}{\sqrt{n}}+1\right)-\frac{n}{2}\underset{n \to +\infty}{\longrightarrow}\frac{x^{4}}{4}+\frac{y^{2}}{4}\]
Moreover the continuity of $f$ and $\phi$ on $D^{\!+}$ gives us
\[f\left(x,\frac{y}{\sqrt{n}}+1\right)\phi\left(\frac{x^{2}}{\sqrt{n}},\frac{y}{\sqrt{n}}+1\right)\ind{y+\sqrt{n}>x^{2}}\ind{x^{2}<\d \sqrt{n}}\ind{|y|<\d\sqrt{n}}\underset{n \to +\infty}{\longrightarrow}f(x,1)\]
Finally the function inside the integral defining $\sqrt{n}e^{n/2}A_{n}$ is dominated by 
\[(x,y) \longmapsto \|f\|_{\infty}\exp\left(-\frac{1}{8}(x^{4}+y^{2})\right)\]
which is independent of $n$ and integrable with respect to the Lebesgue measure on $\R^{2}$. By Lebesgue's dominated convergence theorem, we have
\[\sqrt{n}e^{n/2}A_{n}\underset{n \to +\infty}{\longrightarrow}\int_{\R^{2}}f(x,1)e^{-x^{4}/4}e^{-y^{2}/4}\,dx\,dy=\sqrt{4 \pi}\int_{\R}f(x,1)e^{-x^{4}/4}\,dx\]
We define 
\[\Bl_{\d}=\{\,(x,y) \in D^{\!+}:|x|<\d,\,|y-1|<\d\,\}\]
and
\[B_{n}=\int_{(x^{2}/\sqrt{n},y) \in \Bl_{\d}^{c}} f(x,y)\exp\left(-n\psi\left(\frac{x^{2}}{\sqrt{n}},y\right)\right)\,\phi\left(\frac{x^{2}}{\sqrt{n}},y\right) \ind{\sqrt{n}y>x^{2}}\,dx\,dy\]
Let $\eps=\inf\,\{\,\psi(x,y):(x,y) \in \Bl_{\d}^{c}\}$,
\[|B_{n}|\leq e^{-(n-2)\eps}\|f\|_{\infty}\int_{\R^{2}}\exp\left(-2\psi\left(\frac{x^{2}}{\sqrt{n}},y\right)\right)\phi\left(\frac{x^{2}}{\sqrt{n}},y\right) \ind{\sqrt{n}y>x^{2}}\,dx\,dy\]
The change of variables $(x,y) \longmapsto (x n^{1/4},y)$ yields
\[\sqrt{n}e^{n/2}|B_{n}|\leq e^{2\eps}\|f\|_{\infty}e^{-n(\eps-1/2)}n^{3/4} \int_{\R^{2}}e^{-2\psi(x^{2},y)}\phi(x^{2},y)\ind{x^{2}<y} \,dx\,dy  \]
Lemma~\ref{lemGH} guarantees that $\eps>1/2$ and, using the change of variables given by the function $(x,y)\longmapsto(x,y+x^{2})$, we get
\[\int_{\R^{2}}e^{-2\psi(x^{2},y)}\phi(x^{2},y)\ind{x^{2}<y} \,dx\,dy \leq e \, \left(\int_{\R}e^{-x^{2}}\,dx \right)\left(\int_{0}^{+\infty}\frac{e^{-y}}{\sqrt{y}}\,dy \right) < +\infty\]
Therefore $\sqrt{n}e^{n/2}B_{n}$ goes to $0$ as $n$ goes to $+\infty$. Finally
\begin{multline*}
\int_{\R^{2}}f(x,y)\exp\left(-n\psi\left(\frac{x^{2}}{\sqrt{n}},y\right)\right)\phi\left(\frac{x^{2}}{\sqrt{n}},y\right) \ind{\sqrt{n}y>x^{2}}\,dx\,dy\\
=A_{n}+B_{n}\underset{+\infty}{=} \frac{e^{-n/2}}{\sqrt{n}}\left(\sqrt{4 \pi}\int_{\R}f(x,1)e^{-x^{4}/4}\,dx+o(1)+o(1)\right)
\end{multline*}
If $f=1$, we have
\[\frac{C_{n}}{n^{7/4}n^{(n-3)/2}}\underset{+\infty}{\sim}\sqrt{\frac{4 \pi}{n}}e^{-n/2}\int_{\R}e^{-x^{4}/4}\,dx\]
Hence
\begin{multline*}
\E_{\tilde{\mu}_{n,1}}\left(f\left(\frac{S_{n}}{n^{3/4}},\frac{T_{n}}{n}\right)\right)\underset{n \to +\infty}{\longrightarrow}\int_{\R}f(x,1)\frac{e^{-x^{4}/4}\,dx}{\int_{\R}e^{-u^{4}/4}\,du}\\=\int_{\R^{2}}f(x,y)\left(\frac{e^{-x^{4}/4}\,dx}{\int_{\R}e^{-u^{4}/4}\,du}\otimes \d_{1}(y)\right)
\end{multline*}
Since $(\s X_{n}^{1},\dots,\s X_{n}^{n})$ has the distribution $\widetilde{\mu}_{n,\s}$ if and only if $(X_{n}^{1},\dots,X_{n}^{n})$ has the distribution $\widetilde{\mu}_{n,1}$, we obtain that, under $\widetilde{\mu}_{n,\s}$,
\[\frac{S_{n}}{n^{3/4}} \overset{\Lc}{\underset{n \to +\infty}{\longrightarrow}} \frac{e^{-x^{4}/4\s^{4}}dx}{\int_{\R}e^{-y^{4}/4\s^{4}}\,dy} \qquad \mbox{and} \qquad \frac{T_{n}}{n} \overset{\Lc}{\underset{n \to +\infty}{\longrightarrow}} \s^{2}\]
We also get that $S_{n}/n$ converges in law (thus in probability) to $0$. Finally the ultimate change of variable $y=\sqrt{2}\s x^{1/4}$ implies that
\[\int_{\R}e^{-y^{4}/4\s^{4}}\,dy=2\int_{0}^{+\infty}e^{-y^{4}/4\s^{4}}\,dy=\frac{\s}{\sqrt{2}}\int_{0}^{+\infty}x^{1/4-1}e^{-x}\,dx=\frac{\s}{\sqrt{2}}\,\G\left(\frac{1}{4}\right)\]
This ends the proof of theorem~\ref{theoCVgauss}.

\nocite{Rudin}
\bibliographystyle{plain}
\bibliography{biblio}

\begin{thebibliography}{1}

\bibitem{BTW}
Per Bak, Chao Tang, and Kurt Wiesenfeld.
\newblock Self-organized criticality: An explanation of 1/f noise.
\newblock {\em Phys. Rev. Lett.}, 59:381--384, 1987.

\bibitem{CerfGorny}
Rapha\"el Cerf and Matthias Gorny.
\newblock A {C}urie-{W}eiss model of self-organized criticality.
\newblock {\em preprint}, 2013.

\bibitem{Ellis}
Richard~S. Ellis.
\newblock {\em Entropy, large deviations, and statistical mechanics}.
\newblock Classics in Maths. Springer-Verlag, 2006.

\bibitem{EN}
Richard~S. Ellis and Charles~M. Newman.
\newblock Limit theorems for sums of dependent random variables occurring in
  statistical mechanics.
\newblock {\em Z. Wahrsch. Verw. Gebiete}, 44(2):117--139, 1978.

\bibitem{Feller2}
William Feller.
\newblock {\em An introduction to probability theory and its applications.
  {V}ol. {II}.}
\newblock Second edition. John Wiley \& Sons Inc., 1971.

\bibitem{Rudin}
Walter Rudin.
\newblock {\em Real and complex analysis}.
\newblock McGraw-Hill Book Co., third edition, 1987.

\end{thebibliography}
\addcontentsline{toc}{section}{References}
\markboth{\uppercase{References}}{\uppercase{References}}

\end{document}